\documentclass[11pt,a4paper]{article}

\usepackage{amsmath,amsthm,amsopn,amsfonts,graphicx,amssymb,dsfont,color}
\usepackage{mathrsfs}
\usepackage{color}
\usepackage{hyperref}
\usepackage{mathtools}
\usepackage{enumerate}
\usepackage{a4,a4wide}
\usepackage{soul}
\usepackage[affil-it]{authblk}
    
\newcommand{\NN}{\mathbb{N}}

\newcommand{\RR}{\mathbb{R}}

\newcommand{\dd}{\mathrm{d}}

\newcommand{\eps}{\varepsilon} 
\newcommand{\ep}{\varepsilon}

\newcommand{\correc}[1]{{\textcolor{red}{#1}}}

\newtheorem{theo}{Theorem}[section]

\newtheorem{lem}[theo]{Lemma}
\newtheorem{rema}[theo]{Remark}
\newtheorem{defi}[theo]{Definition}

\title{\bf Infinitely many saturated travelling waves for a degenerate Fisher-KPP equation not in divergence form}

\author{Matthieu Alfaro\thanks{\texttt{matthieu.alfaro@univ-rouen.fr}\\ Université de Rouen Normandie, CNRS, Laboratoire de Mathématiques Raphaël Salem (LMRS), Saint-Étienne-du-Rouvray, France}, Maxime Herda\thanks{\texttt{maxime.herda@inria.fr}\\ Univ. Lille, CNRS, Inria, UMR 8524 - Laboratoire Paul Painlevé, F-59000 Lille, France}, and Andrea Natale\thanks{\texttt{andrea.natale@inria.fr}\\ Univ. Lille, CNRS, Inria, UMR 8524 - Laboratoire Paul Painlevé, F-59000 Lille, France}}

\date{}

\begin{document}
\maketitle

\begin{abstract} We consider an epidemic model with distributed-contacts. When the contact kernel concentrates, one formally reaches a very degenerate Fisher-KPP equation with a diffusion term that is not in divergence form. We make an exhaustive study of its travelling waves. For every admissible speed, there exist not only a unique non-saturated (smooth) wave but also infinitely many saturated (sharp) ones. Furthermore their tails may differ from what is usually expected. These results are thus in sharp contrast with their counterparts on related models.\\


\end{abstract}

\section{Introduction}\label{s:intro}

Starting from a $SI$ epidemic model with so-called distributed-contacts, we show the relevance of a very degenerate (since not in divergence form) reaction-diffusion equation of the Fisher-KPP type for the infectious density $I=I(t, x)$, namely
\begin{equation}
\label{edp}
\partial _t I=(1-I)\Delta I +I(1-I), \quad t>0,\, x \in \mathbb R^N,
\end{equation}
where $N\geq 1$. We focus on the existence and properties
of planar travelling waves for \eqref{edp}. We show that, for any admissible speed, there exist not only a unique non-saturated (smooth) wave but also infinitely many saturated ones. This is in sharp contrast with existing results on related models. The tails of these waves also exhibit unexpected behavior.

\medskip

Compartmental epidemics models were introduced in the seminal work  of Kermack and McKendrick \cite{Ker-Mac-27}. The simplest $SI$ model consists of a ODE system for $S=S(t)$, $I=I(t)$,  the number of Susceptible and  Infectious, uses a law of mass action and writes 
 \begin{eqnarray}\label{SIrecovery}
\begin{cases} \frac{dS}{dt}=-\beta SI &\quad t>0,  \vspace{5pt}\\
\frac{dI}{dt}=\beta SI-\mu I &\quad t>0,
\end{cases}
\end{eqnarray}
where $\beta>0$ is the constant transmission rate, $\mu \geq 0$ the constant recovery rate. However spatial effects, which have a very determining effect on the propagation of epidemics, are neglected in this simple model. To fill this gap, Kendall \cite{Ken-65} and Mollison \cite{Mol-72} have allowed (see also   \cite{Thi-77}, \cite{Die-78},  \cite{Med-Kot-03})  spatially distributed-contacts between individuals. In this framework, in absence of recovery ($\mu=0$) and letting $\beta=1$ up to changing the time scale,  $S=S(t,x)$, $I=I(t,x)$ solve the integro-differential system
 \begin{eqnarray}\label{integro-diff-SI}
\begin{cases} \partial _t S=- S \displaystyle \int_{\mathbb R^N} K(x,y)I(t,y)\,dy \quad & t>0,\, x\in \mathbb R^N, \vspace{5pt}\\
\partial _t I= S \displaystyle \int_{\mathbb R^N} K(x,y)I(t,y)\,dy \quad & t>0,\, x\in \mathbb R^N, 
\end{cases}
\end{eqnarray}
with $K(x,y)$ denoting the density function for the proportion of infectious at position $y$ that contact susceptibles at position $x$. Since $S+I$ is independent of time, say equal to $1$, we may rewrite 
\begin{equation}\label{eq-integro-diff}
\partial _t I=(1-I)(J*I-I)+ I(1-I), \quad t>0, \,  x\in \mathbb R ^N,
\end{equation}
where we have assumed  $K(x,y)=J(x-y)$ with $J$ a probability density on  $\mathbb R^N$, meaning that  contacts are {\it homogeneous} in space. However, in some situations, nonlocal diffusion operators can be approximated by local ones. This fact is well-known, see e.g. \cite[Chapter VI, subsection 6.4]{Bur-00-book}, and can be understood from a simple formal Taylor expansion. Indeed, assuming that $J$ is radial (note that weaker conditions such as component-wise symmetry would be enough) and has a finite second moment $J_2\coloneqq\int_{\mathbb R ^N}|z| ^2J(z)\,dz<+\infty$, replacing the kernel $J(z)$ by the focused kernel
$$
J_\ep(z)\coloneqq\frac{1}{\ep^{N}}J\left(\frac  z \ep\right), \quad 0<\ep\ll 1,
$$
we see that
\begin{eqnarray*}
(J_\ep *I-I)(t,x)&=&\int_{\mathbb R ^N}\frac{1}{\ep ^N}J\left(\frac{y-x}{\ep}\right)(I(t,y)-I(t,x))\,dy\\
&=&\int_{\mathbb R ^N}J(z)(I(t,x+\ep z)-I(t,x))\,dz\\
&\approx&  \ep ^2 \frac{J_2}{2N} \Delta I(t,x), \quad \text{ as } \ep \to 0.
\end{eqnarray*}
For a rigorous statement one may refer to \cite[Theorem 1.24]{And-10}. Let us also note that the case of nonlocal {\it heterogeneous} diffusion has recently been explored in \cite{Alf-Gil-Kim-Pel-Seo-22}.
Hence, starting from a $SI$ model with distributed-contacts, and formally passing to the focusing kernel limit, we have reached a degenerate Fisher-KPP equation of the form \eqref{edp}.

\medskip

Since the introduction of spatial effects in epidemics models, the issue of travelling wave solutions  has attracted much attention. We may refer, among many others, to the works \cite{Atk-Reu-76}, \cite{Bro-Car-77}, \cite{Bar-77},  \cite{Aro-77}, \cite{Thi-77}, \cite{Die-78}, or \cite{Hos-Ily-95}. Very recently, a renewed interest for these models, both from the modelling and the mathematical analysis point of view, has emerged. Let us mention for instance the focus on heterogeneities \cite{Duc-Gil-14}, \cite{Duc-20}, \cite{Duc-22}, the effect of  fast lines of diffusion \cite{Ber-Roq-Ros-21}, the spread of epidemics on graphs \cite{Bes-Fay-21}, the spread of several variants of a disease \cite{Duc-Nor-23}, \cite{Bur-Duc-Gri-23}, \cite{Fay-Roq-Zha-23}, etc.

\medskip

Planar travelling wave solutions are  particular solutions (of reaction-diffusion equations) describing the transition at a constant speed $c$ from one stationary solution to another one. In the case of the classical Fisher-KPP equation $\partial _t I=\Delta I +I(1-I)$,
it has long been known \cite{Fis-37}, \cite{Kol-Pet-Pis-37},   \cite{Aro-Wei-78}, that they exist if and only if $c\geq c^*\coloneqq2$.  Moreover, they are very accurate to describe the long time behavior of the Cauchy problem, the tails of the initial data selecting the speed $c\geq c^*$. In particular, the minimal speed $c^*=2$ corresponds to the
so-called {\it spreading speed} of propagation for the Cauchy problem with compactly supported data. In the Fisher-KPP framework, let us also mention the construction of travelling wave solutions in presence of density-dependent diffusion \cite{Eng-85}, possibly degenerate \cite{San-Mai-94, San-Mai-95}, \cite{Mal-Mar-03},  \cite{Dra-Tak-21}. However, much less is known when the equation is, as \eqref{edp}, not in divergence form.  We  mention the works \cite{Wan-Yin-03}, \cite{Yin-Jin-09, Yin-Jin-09-bis}, mainly concerned with diffusion of the form $I^m\Delta I$ (or variants involving the $p$-Laplacian), meaning that singularities are caused by $I\equiv 0$ (the unstable equilibrium of the logistic equation), which is in sharp contrast with \eqref{edp} where singularities are caused by $I\equiv 1$ (the stable equilibrium of the logistic equation). 

\medskip

In this work, we are thus concerned with the existence of travelling wave solutions for \eqref{edp}. For any admissible speed $c\geq c^*\coloneqq2$, we will show that there is a travelling wave solution with value in $(0,1)$, which is quite expected. However, the equation being not in divergence form, we will also show that, for any admissible speed $c\geq c^*\coloneqq2$, there are infinitely many travelling waves saturating at value 1 on some semi-infinite interval. This phenomenon is quite unusual and raises many questions regarding the Cauchy problem, in particular concerning its well-posedness and long time behavior.

\section{Main results}\label{s:results}

Let us propose a notion of weak solution for equation \eqref{edp}. First, we impose the bound $I\leq 1$ for the equation not to become anti-diffusive.
Next, we rewrite the diffusion term $(1-I)\Delta I$ as  $\nabla\cdot((1-I)\nabla I) + |\nabla I|^2$ and, classically, multiply by test functions and integrate by parts, see  \cite{Ber-Dal-Ugh-90} for a similar formulation.

\begin{defi}[Weak solution]\label{defi:weak}
A weak solution of \eqref{edp} is a function $I\in L^\infty(\mathbb{R}\times\mathbb{R}^N)\cap L^2_\text{loc}(\mathbb{R}, H^1_\text{loc}(\mathbb{R}^N))$ such that $I\leq 1$ and
\begin{equation}\label{weak-formu}
\int_{\mathbb{R}\times\mathbb{R}^N} \left(I\partial_t\varphi - (1-I)\nabla I\cdot\nabla\varphi +|\nabla I|^2\varphi + I(1-I)\varphi\right)(t,x) \,dt dx = 0,
\end{equation}
 for all compactly supported $\varphi\in\mathcal{C}^{1}(\mathbb{R}\times\mathbb{R}^N)$.
\end{defi}

As far as the Cauchy problem is concerned, the well-posedness (existence and uniqueness) is a delicate issue because of the non divergence form of the equation, see \cite{All-83}, \cite{Ugh-84}, \cite{Dal-Luc-87}, \cite{Ber-Ugh-90}, \cite{Ber-Dal-Ugh-90}. In \cite{Ugh-84} Ughi, in the one dimensional case, has worked in the framework of solutions that are {\it Lipschitz} in space. Using a formulation in the spirit of Definition \ref{defi:weak}, she proves existence by some approximation/regularization procedure, but highlights a non uniqueness result. Next, she proposes another weak formulation based on
$$
\partial _t (\log (1-I))=\Delta (1-I)-I,
$$
obtained by first dividing the equation by $1-I$. It turns out that this improved formulation ensures uniqueness. However, the support of $1-I$ does not depend on time. It is challenging to define a notion of weak solution for the Cauchy problem not only ensuring uniqueness but also allowing this support to move, hence capturing the saturated fronts constructed in the present paper. We underline that these saturated fronts are not Lipschitz in space and this fact should play a key role. We hope to address this issue in a future work.

 Now we define the notion of a planar travelling wave for \eqref{edp} (see Remark \ref{rem:la-def-est-bonne}). 
 
\begin{defi}[Travelling wave profile]\label{defi:profile}
A travelling wave profile is a couple $(z^*,u)$ made of a $-\infty\leq z^*<+\infty$ and a function $u:(z^*,+\infty)\to(0,1)$ of class $C^2$ on $(z^*,+\infty)$, such that $u(z^*)=1$, $u(+\infty)=0$.
\begin{enumerate}
    \item [(i)] If $z^*=-\infty$, the wave is said to be {\it \bf non-saturated}.
    \item [(ii)] If $z^*\neq - \infty$, the wave is said to be {\it \bf saturated}.
\end{enumerate}
\end{defi}

\begin{defi}[Travelling wave solution to \eqref{edp}]\label{defi:travwave}
A travelling wave solution to \eqref{edp} propagating in the direction $e\in \mathbb S^{N-1}$ is a triplet $(c,z^*,u)$, made of a speed $c\in \RR$ and a travelling wave profile $(z^*,u)$, such that 
\begin{equation}\label{eq:correspIu}
I(t,x) \coloneqq \left\{\begin{array}{ll}u(x\cdot e-ct)& \text{if}\  x\cdot e-ct > z^*,\\1&\text{if}\ x\cdot e-ct \leq z^*\end{array}\right.
\end{equation}
is a weak solution to \eqref{edp} in the sense of Definition~\ref{defi:weak}.
\end{defi}

Our main result is the following exhaustive description of travelling waves.

\begin{theo}[Travelling waves]\label{th:tw} There is no travelling wave solution to \eqref{edp} with $c<2$. Next, let $c\geq 2$ be fixed. Then there exist a non-saturated and infinitely many saturated waves normalized by $u(0)=\frac 12$. More precisely, defining
$$
 \lambda^-\coloneqq\frac{c-\sqrt{c^2-4}}{2}, \quad  \lambda^+\coloneqq\frac{c+\sqrt{c^2-4}}{2},
$$
there are constants $C_i>0$ such that for any $0<\ep\ll 1$ the following hold.
\begin{enumerate}
\item [(i)] There exists a unique non-saturated wave whose profile is denoted \correc{by} $u_{NS}$. It satisfies
\begin{equation}
\label{decay-0-non-saturated}
C_1\,e^{-\lambda^-z}\leq u_{NS}(z)\leq  C_2\,e^{-(\lambda^--\ep) z}, \quad \text{ as } z\to +\infty, 
\end{equation}
\begin{equation}
\label{decay-1-non-saturated}
C_3\,e^{\frac 1 c z}\leq 1- u_{NS}(z)\leq C_4\,e^{\frac 1 {c+\ep} z}, \quad \text{ as } z\to-\infty.
\end{equation}
Last, there is $z_0>0$ such that $u''<0$ on $(-\infty,z_0)$ while $u''>0$ on $( z_0,+\infty)$.
\item [(ii)] There exist infinitely many saturated waves whose profiles are denoted by $u_S$. They satisfy
\begin{equation}
\label{behavior-1-saturated}
1-u_S(z) \sim -c(z-z^*)\log (z-z^*), \quad \text{ as } z\searrow z^*.
\end{equation}
Among them
\begin{enumerate}
    \item [(a)] there are infinitely many ones for which there are $z_1<0<z_2$ such that $u_S''>0$ on $(z^*,z_1)$, $u_S''<0$ on $(z_1,z_2)$, $u_S''>0$ on $(z_2,+\infty)$. They satisfy
        \begin{equation}
        \label{decay-0-saturated}
        C_5\,e^{-\lambda^-z}\leq u_{S}(z)\leq  C_6\,e^{-(\lambda^--\ep) z}, \quad \text{ as } z\to +\infty.
        \end{equation}
    \item [(b)] there is one such that $u_S''>0$ on $(z^*,0)\cup(0,+\infty)$ and $u_S''(0)=0$. It  satisfies \eqref{decay-0-saturated}.
    \item [(c)] there are infinitely many ones such that $u_S''>0$ on $(z^*,+\infty)$.  
     Infinitely many of them satisfy
     \begin{equation}
     \label{decay-0-saturated-heavy}
         C_7\,e^{-(\lambda^-+\ep)z} \leq u_{S}(z)\leq  C_{8}\, e^{-(\lambda^- -\ep)z}, \quad \text{ as } z\to +\infty.
    \end{equation}
     At least one other satisfies 
      \begin{equation}
\label{decay-0-saturated-light}
C_9\,e^{-(\lambda ^++\ep) z}\leq u_{S}(z)\leq  C_{10}\,e^{-\lambda ^+ z}, \quad \text{ as } z\to +\infty.
\end{equation} 
\end{enumerate}
\end{enumerate}
\end{theo}

Observe first that the set of admissible speeds, namely $[2,+\infty)$, is the same as that for the classical Fisher-KPP travelling waves. However, the above results are in very sharp contrast with previous results on related problems. 

Indeed, and first of all,  for a given admissible speed $c\geq 2$,  {\it two} different kinds of waves, non-saturated and saturated, co-exist. This is already very unusual and, of the top of it, the saturated waves are {\it infinitely many}. 

Also, from our construction, any two of the waves of Theorem \ref{th:tw} cross each other exactly once, namely at $z=0$. Furthermore, see the shooting argument with $\alpha$ in \eqref{cauchy-un-demi}, 
$$
u_{NS} > u_{S} \text{ of type $(a)$}>u_{S} \text{ of type $(b)$}>u_{S} \text{ of type $(c)$}\quad \text{  on } (0,+\infty), 
$$
and reversed order on $(-\infty,0)$. Also from the shooting argument we use in Section \ref{s:tw}, the rightmost points $z^*$ where the saturated waves are equal to 1 span $(-\infty,z^*_{max}]$ for some $z^*_{max}<0$, see Figure \ref{fig:waves} for a picture of the situation.

\begin{figure}[ht]
\begin{center}
  \includegraphics[scale=.75]{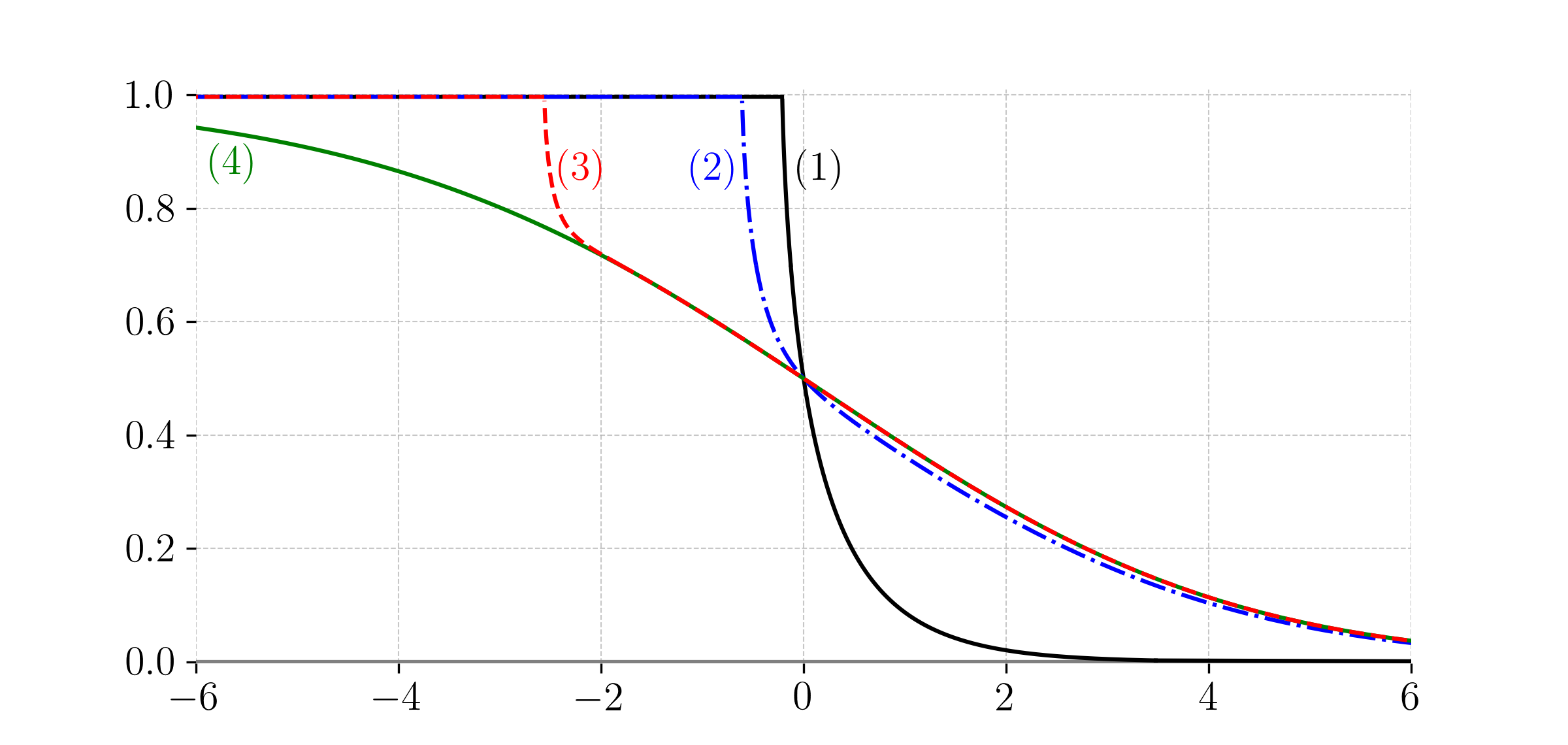}
    \caption{Some travelling waves of Theorem \ref{th:tw} (for a given speed $c=2.1$). The profile (1), (2) and (3) are saturated waves. Specifically, (1) and (2) are profiles of type (c) (convex), whereas (3) is a profile of type (a) (convex-concave-convex). Profile (4) corresponds to the unique non-saturated wave.}
   \label{fig:waves}
   \end{center}
\end{figure}

Another striking phenomenon concerns the {\it right tails} of these waves. Indeed, when $c>2$ (supercritical speeds), \eqref{decay-0-non-saturated} and \eqref{decay-0-saturated-light} indicate that the decay towards zero of {\it some} $u_{NS}$ and $u_S$ are not the same: since then $\lambda^-<\lambda^+$, we even have
\begin{equation}\label{truc-de-ouf}
\text{ if }\; c>2\; \text{ then }\; u_S(z)\ll u_{NS}(z) \text{ as } z\to +\infty.
\end{equation}
Let us recall that if we linearize at $z=+\infty$ the travelling wave equation for the  classical Fisher-KPP equation, namely $u''+cu'+u(1-u)=0$, we get $u''+cu'+u=0$, which obviously admits two exponential solutions $e^{-\lambda^-z}$ and $e^{-\lambda^+z}$, connecting $+\infty$ as $z\to-\infty$ to 0 as $z\to+\infty$. However, when considering the full nonlinear equation with $u(-\infty)=1$, the right tail is of the magnitude $e^{-\lambda^-z}$, namely
$$
u_{classical}(z)\sim \begin{cases}
C e^{-\lambda^-z}  & \text{ if } c>2,\\
Cz e^{-z} & \text{ if }c=2,
\end{cases}, \quad \text{ as } z\to +\infty.
$$
In other words, the behavior $e^{-\lambda^+z}$ is erased by the \emph{a priori} condition $u' = 0$ at $u=1$ in the classical Fisher-KPP equation but can still persist in the degenerate case under consideration, see \eqref{decay-0-saturated-light}. Hence, while the tail behaviors \eqref{decay-0-saturated} and \eqref{decay-0-saturated-heavy} appear comparable to those in the classical Fisher-KPP situation, the tail behavior \eqref{decay-0-saturated-light} is in sharp contrast with the classical case.

Note that \eqref{truc-de-ouf} is true for the $u_{S}$ satisfying \eqref{decay-0-saturated-light} (at least one), but this may still hold for any of the constructed $u_{S}$'s. Understanding this would require very refined asymptotics of $h(r)$ as $r\to 0$ (see Section \ref{s:tw} for the definition of $h$ and details)
 and is left as an open question.
 
Similarly, determining if there is {\it exactly one} or {\it infinitely many} waves satisfying \eqref{decay-0-saturated-light} is far from obvious and would require  refined estimates of $h(r)$ as $r\to 0$ (see Section \ref{s:tw} for details).

The {\it left tail} of the non-saturated wave $u_{NS}$ is also unusual. Indeed, in the  classical Fisher-KPP equation, as seen from linearizing at $z=-\infty$ and as well-known,
$$
1-u_{classical}(z) \sim Ce^{\frac{-c+\sqrt{c^2+4}}{2}z}, \quad  \text{ as } z\to -\infty.
$$
However, since $\frac 1 c>\frac{-c+\sqrt{c^2+4}}{2}$, \eqref{decay-1-non-saturated} indicates that the left tail of the non-saturated wave is much steeper, namely
$$
1-u_{NS}(z)\ll 1-u_{classical}(z), \quad \text{ as } z\to -\infty.
$$
The reason for \eqref{decay-1-non-saturated} can be understood as follows: as $z\to -\infty$, meaning $u\to 1$, the degenerate travelling wave equation $(1-u)u''+cu'+u(1-u)=0$ is, at least formally, approximated by the first order ODE $cu'+(1-u)=0$ so that  $1-u(z)$ decays like $e^{\frac 1 c z}$.

Obviously, these various phenomena are induced by the non-divergence form of the equation and raise further questions not only on the well-posedness (as mentioned above)  but also on the long time behavior of the Cauchy problem associated with \eqref{edp}. Indeed, the exponential right tail of a front-like initial datum may not decide alone the spreading speed, as in the classical case. We believe that if $0<u_0<1$ on $\RR$ then the spreading speed is the wave speed of $u_{NS}$ having the right tail \lq\lq matching'' with that of $u_0$. However, if $u_0\equiv 1$ on some interval,  the picture is less clear. This would deserve further investigations. 

\medskip

The rest of the paper is devoted to the proof of Theorem \ref{th:tw}.  The strategy consists in transforming in subsection \ref{ss:reaching} the problem into a first order singular ODE Cauchy problem, studied in details in subsections \ref{ss:Cauchy} and \ref{ss:boundary}. Last, we conclude in subsection \ref{ss:back}.

\section{Travelling wave solutions}\label{s:tw}

By formally plugging the ansatz $I(t,x) = u(x\cdot e-ct)$ into \eqref{edp}, it is clear that the profile $(z^*,c,u)$ should satisfy the equation
\begin{equation}
    \label{eq-tw}
    (1-u)u''+cu'+u(1-u)=0, \quad \text{ on }\, (z^*,+\infty).
\end{equation}

We start with the following basic facts about \eqref{eq-tw}.

\begin{lem}[A priori facts]\label{lem:a-priori} Let $(z^*,u)$ be a travelling wave profile in the sense of Definition~\ref{defi:profile} which, for some $c\in \RR$, satisfies \eqref{eq-tw}.  Then
\begin{enumerate}
    \item [(i)] $u'<0$ on $(z^*,+\infty)$.
    \item [(ii)] $u'(+\infty)=0$, $(1-u(z))u'(z)\to 0$ as $z\to z^*$, and $u'\in L^2(z^*,+\infty)$.
    \item [(iii)] $c>0$.
\end{enumerate}
\end{lem}

\begin{proof} Since $u$ takes value in $(0,1)$, it follows from  equation \eqref{eq-tw} that any critical point $z$ is such that $u''(z)<0$. Hence if there is $z_0\in(z^*,+\infty)$ such that $u'(z_0)=0$ the boundary condition $u(z^*)=1$ cannot be satisfied. This proves $(i)$.

Since $(1-u)u''=((1-u)u')'+u'^2$, integrating the equation over $(z,t)$ yields
\begin{equation}\label{eq-integree}
(1-u(z))u'(z)=(1-u(t))u'(t)+c(u(t)-u(z))+\int_z^t(u'^2+u(1-u))(s)\,ds.
\end{equation}
Since $u'^2+u(1-u)>0$, the right hand side admits a limit $\ell\in \RR \cup\{+\infty\}$ as $z\to z^*$, and so does $(1-u(z))u'(z)$. From $(i)$, $\ell \in(-\infty,0]$. Assume by contradiction  $\ell < 0$ so that 
\begin{equation}
\label{equiv-faux}
((1-u)^2)'(z)\sim -2\ell, \quad\text{  as }z\to z^*. 
\end{equation}
 If $z^*=-\infty$, it immediately follows, by integration of \eqref{equiv-faux}, that $(1-u)^2(z)$ is unbounded as $z\to-\infty$, which is absurd. Let us now consider the case $-\infty<z^*$. In this case, by integration of \eqref{equiv-faux}, we reach, as $z\to z^*$, $(1-u)(z)\sim \sqrt{-2\ell}\sqrt{z-z^*}$ which, in turn, provides  $u'(z)\sim\frac{\ell}{\sqrt{-2\ell}}\frac{1}{\sqrt{z-z^*}}$. Plugging these two informations into the travelling wave equation \eqref{eq-tw}, we then obtain, as $z\to z^*$,  $u''(z)\sim \frac{c}{2(z-z^*)}$ (note that the case $c=0$ is obviously excluded by the equation) which, after integration, provides $u'(z)\sim \frac{c}{2}\log (z-z^*)$, which contradicts a previous equivalent. Hence, in any case, $\ell=0$ that is  $(1-u(z))u'(z)\to 0$ as $z\to z^*$. The limit $u'(+\infty)=0$ may be obtained by similar arguments, or also via classical elliptic estimates.

As a result, we may let $z\to z^*$ and $t\to +\infty$ in \eqref{eq-integree} to reach
$$
c=\int_{z^*}^{+\infty}(u'^2+u(1-u))(s)\,ds,
$$
which concludes the proof of $(ii)$ and $(iii)$.
\end{proof}

\begin{rema}\label{rem:la-def-est-bonne}The argument to prove $(i)$ in Lemma \ref{lem:a-priori} actually excludes the possibility of travelling waves not having the form prescribed by Definition \ref{defi:profile}. For instance if $u\equiv 1$ on some $[z_*,z^*]$ with $-\infty<z_*<z^*<+\infty$ and $u<1$ on some left neighborhood of $z_*$, then we would have $u'>0$ on the left of $z_*$, hence making the boundary condition $u(-\infty)=1$ unreachable.  
\end{rema}

Thanks to the previous a priori facts we can show, in the following lemma, that we may indeed restrict our attention to the study of \eqref{eq-tw} in order to study travelling wave solutions to \eqref{edp}.

\begin{lem}[Equivalent formulations] Let $(z^*,u)$ be a travelling wave profile in the sense of Definition~\ref{defi:profile}. Let $c\in \RR$ be given. Then the following are equivalent.
\begin{itemize} 
\item[(a)] The function $I$ defined by \eqref{eq:correspIu} is a weak solution to \eqref{edp} in the sense of Definition~\ref{defi:weak}.
\item[(b)] The function $u$ satisfies \eqref{eq-tw}.
\end{itemize}
\end{lem}

\begin{proof} Let us assume $(a)$. Without loss of generality we can assume that $e$ is the basis vector associated with the first component $x_1$. Assume also $c\neq 0$. The $c=0$ case can be dealt with as a simpler adaptation of what follows. Plugging \eqref{eq:correspIu} into the weak formulation of Definition~\ref{defi:weak}, one obtains
\[
\int_{x_1-ct>z^*} u\partial_t\varphi - (1-u)u'\partial_{x_1}\varphi +|u'|^2\varphi + u(1-u)\varphi \,dt dx+ \int_{x_1-ct\leq z^*}\partial_t\varphi  \,dt dx= 0,
\]
where $u$ and $u'$ are evaluated at $x_1-ct$ in the first integral. Then we perform the linear change of variable $y_1 =  x_1-ct$, $y_i = x_{i-1}$ for $i\in\{2,\dots,N+1\}$ and choose as a test function the tensor product $\varphi(t,x_1,\dots,x_N) = \psi(y_1)\chi(y_2,\dots,y_N)$ with compactly supported $\psi$ and $\chi$ such that $\int\chi\neq0$. Using that $\partial_t\varphi  = -c\psi'\chi$ and $\partial_{x_1} \varphi = \psi'\chi+\psi\partial_{x_1}\chi$, one easily finds

\begin{equation}\label{weak_u}
c\psi(z^*) + \int_{z^*}^{+\infty} \left[cu\psi' + (1-u)u'\psi' - |u'|^2\psi-u(1-u)\psi\right]\dd z= 0\,,
\end{equation}
which yields \eqref{eq-tw} after integrating by parts for $\psi$ compactly supported in $(z^*,+\infty)$.

Let us assume $(b)$.  If $u$ satisfies \eqref{eq-tw}, then by Lemma~\ref{lem:a-priori}, the corresponding $I$ defined by \eqref{eq:correspIu} satisfies $I\in L^\infty(\mathbb{R}\times\mathbb{R}^N)\cap L^2_\text{loc}(\mathbb{R}, H^1_\text{loc}(\mathbb{R}^N))$. It remains to show that the weak formulation \eqref{weak-formu} is indeed satisfied. For this, one can plug the expression of $I$ into it and check that it equivalently holds if for all smooth $\psi$ supported in $\overline{(z^*, +\infty)}$,  one has \eqref{weak_u}. Observe that unlike the previous implication here one has to allow test functions such that $\psi(z^*)\neq0$ when $z^*$ is finite. However this holds again by integration by parts using Lemma~\ref{lem:a-priori} for the boundary terms and the strong formulation \eqref{eq-tw}.
\end{proof}

\subsection{Reaching a first order ODE}\label{ss:reaching}

We now aim at constructing travelling wave solutions. Our approach  consists in transforming the problem into a first order ODE, a strategy already proved to be very relevant in related situations, see \cite{San-Mai-94, San-Mai-95}, \cite{Mal-Mar-03}, \cite{Dra-Tak-21} and the references therein. However, because of the degeneracy of the considered equation not in divergence form, the reached ODE is singular, in sharp contrast with \cite{Eng-Gav-San-13} or  \cite{Dra-Tak-21}, and infinitely many waves with the same speed actually co-exist.

\medskip

From Lemma \ref{lem:a-priori} $(i)$, $u$ is a $C^2$ bijection of   $(z^*,+\infty)$ onto $(0,1)$, and we may thus look for $u^{-1}:(0,1)\to (z^*,+\infty)$. Obviously
\begin{equation}
    \label{der-u-moin-un}
    (u^{-1})'(u(z))=\frac{1}{u'(z)}.
\end{equation}
Next, we define
\begin{equation}\label{h-def}
    h\coloneqq (u')^2\circ u^{-1}\,,
\end{equation}
where $\circ$ denotes the composition of functions. By \eqref{der-u-moin-un}, we find 
\[
h' = 2u'' \circ u^{-1}
\]
and replacing the expression for $u''$ derived from \eqref{eq-tw}, we obtain
\[
\frac{dh}{dr} (r) = -\frac{2c}{1-r} u'\circ u^{-1}(r) -2r.
\]
Thus we look for $h:(0,1)\to (0,+\infty)$ solving
 \begin{equation*}\label{edo-h}
\begin{cases} \displaystyle\frac{dh}{dr}=\frac{2c}{1-r}\sqrt{h^+} -2r \quad  & 0<r<1, \vspace{5pt}\\
h(0)=0,
\end{cases}
\end{equation*}
where $h^+ \coloneqq \max(h,0)$ and where the initial datum $h(0)=0$ is a direct consequence of the a priori fact  $u'(+\infty)=0$ proved in Lemma \ref{lem:a-priori} $(ii)$. Note also that the a priori fact  $(1-u(z))u'(z)\to 0$ as $z\to z^*$ proved in Lemma \ref{lem:a-priori} $(ii)$ enforces $(1-r)^2h(r)\to 0$ as $r\nearrow 1$ for any solution.

\subsection{A Cauchy problem}\label{ss:Cauchy}

In this subsection we thus investigate the Cauchy problem
 \begin{equation}\label{h-Cauchy}
\begin{cases} \displaystyle \frac{dh}{dr}=\frac{2c}{1-r}\sqrt{h^+} -2r \quad 0<r<1, \vspace{5pt}\\
h(0)=0.
\end{cases}
\end{equation}
By Cauchy-Peano theorem, there exists at least  one local solution, defined on some $[0,r^*)$ with $0<r^*\leq 1$ (anyway, observe that $r\in[0,1)\mapsto -r^2$ is a global solution!). 

If a solution vanishes at some $r_0>0$ then, from the equation, $h'(r_0)=-2r_0<0$. As a result a solution can only be positive on $(0,r^*)$, or positive on some $(0,r_0)$-negative on $(r_0,r^*)$, or negative on $(0,r^*)$ (and in this case this has to be $r\mapsto -r^2$).

We  claim that any local solution is actually global. Indeed,  
 $h'\geq -2r$ provides $h(r)\geq -r^2$, while $h'\leq \frac{2c}{1-r}\sqrt{h^+}$ provides $h(r)\leq c^2\left(\log(1-r)\right)^2$. Hence
\begin{equation}\label{log-dessus}
-r^2\leq h(r)\leq c^2\left(\log(1-r)\right)^2,
\end{equation}
which prevents blow-up at $r^*<1$. Hence $r^*=1$ which proves the claim. 

If a solution to \eqref{h-Cauchy} further satisfies $h>0$ on $(0,1)$, then this corresponds to a travelling wave through   subsection \ref{ss:reaching}.

We now discuss  the value of $c>0$. 

\begin{lem}[Small speeds]\label{lem:small-speeds} If $0<c<2$ then the only solution of \eqref{h-Cauchy} is $r\mapsto -r^2$.
\end{lem}

\begin{proof} Let $0<c<2$ and let us choose $\eps>0$ small enough that $c<\frac{2}{1+\eps}$. Next select  $r_\eps>0$ small enough so that
\begin{equation}\label{eq:corde}
-\log(1-r)\leq(1+\eps)r\quad\text{for } r\in(0,r_\eps)\,.  
\end{equation}
Assume by contradiction that there is a solution such that $h>0$ on $(0,r_\eps)$. It satisfies
\[
\frac{d \sqrt{h}}{dr} (r)=\frac{c}{1-r}-\frac{r}{\sqrt{h(r)}}, \quad 0<r<r_\eps.
\]
Integrating this and using \eqref{eq:corde}, we obtain
\begin{equation}\label{eq:ineg}
 \sqrt{h(r)}\leq c(1+\eps)r - \int_0^r\frac{s}{\sqrt{h(s)}}\dd s, \quad 0<r<r_\eps.
\end{equation}
In particular $\sqrt{h(r)} \leq c(1+\eps)r$, which can be plugged into \eqref{eq:ineg} to obtain 
$$
\sqrt{h(r)}\leq \left(c(1+\eps)-\frac{1}{ c(1+\eps)}\right)r.
$$
By bootstrapping the bound one finds that 
\begin{equation}\label{eq:bound}
 \sqrt{h(r)} \leq M_{n+1}r, \quad 0<r<r_\eps,
 \end{equation}
 as long as $M_n$ defined iteratively by
\begin{equation}\label{eq:Mn}
M_{n+1} = c(1+\eps) - \frac{1}{M_n},\qquad M_0 = c(1+\eps),
\end{equation}
is nonnegative. Observe that since $c<\frac{2}{1+\eps}$, there is $\alpha<1$ such that $c(1+\eps) - x^{-1}\leq \alpha x$ for all $x>0$. Therefore $M_{n}\leq M_0\alpha^n$ and in particular there is a first $n_0\in\NN$ such that $M_{n_0}<\frac 1 {c(1+\ep)}$. Then $M_{n_0+1}<0$ and \eqref{eq:bound} is a contradiction. Hence, if $c<2$, the only solution is  negative in a neighborhood of $0$ which leaves only $h:r\mapsto -r^2$.
\end{proof}

We now work with $c\geq 2$. We define the relevant quantities
\begin{equation}
\label{def-lambda}
\lambda^-\coloneqq\frac{c-\sqrt{c^2-4}}{2}, \quad \lambda^+\coloneqq\frac{c+\sqrt{c^2-4}}{2},
\end{equation}
solving $\lambda^2-c\lambda+1=0$, and note that $\frac 1 c <\lambda^-\leq \lambda ^+<c$, with equality if and only if $c=2$.

As a preliminary result let us describe some sub- and supersolutions of \eqref{h-Cauchy}. These results will be used repeatedly in the rest of the analysis. The proof 
of the lemma consists in straightforward computations and elementary inequalities and is omitted.

\begin{lem}[Sub- and supersolutions]\label{lem:sub-super} One has the following properties.
\begin{itemize}
\item[(a)] The function $r\mapsto c^2\left(\log(1-r)\right)^2$ is a strict supersolution  of \eqref{h-Cauchy} on $(0,1)$.
\item[(b)] If $\lambda\in [\lambda^-, \lambda^+]$, then $r\mapsto \lambda^2r^2(1-r)^2$ is a strict subsolution of \eqref{h-Cauchy} on $(0,1)$. 
\item[(c)] If $\lambda\in (0, \lambda^-)\cup(\lambda^+, +\infty)$, then  $r\mapsto \lambda^2r^2(1-r)^2$ is a strict supersolution of \eqref{h-Cauchy} on $(0,r_\lambda)$ for some $r_{\lambda}\in(0,1)$.
\item[(d)] If $c\in(2, \frac{3\sqrt{2}}{2})$, $\alpha>0$ and $\beta\in((\lambda^{-})^{-2}-1, 1)$, then  $r\mapsto (\lambda^-)^2r^2(1+\alpha r^{\beta})$ is a strict supersolution of \eqref{h-Cauchy} on $(0,r_{\alpha, \beta})$ for some $r_{\alpha,\beta}\in(0,1)$.
\end{itemize}
\end{lem}

We start by constructing  a \lq\lq large'' solution, in the sense that $h(1)=+\infty$, corresponding to a saturated wave.

\begin{lem}[A first \lq\lq large''  solution] \label{lem:large} Assume $c\geq  2$. Then the Cauchy problem \eqref{h-Cauchy} has a  solution $H$ satisfying
\begin{equation}\label{H-coincee}
(\lambda^+)^2 r^2(1-r)^2< H(r)\leq c^2\left(\log(1-r)\right)^2, \quad 0<r<1.
\end{equation}
Furthermore, this solution is increasing on $[0,1)$ and satisfies $H(1)=+\infty$, more precisely
\begin{equation}\label{eq-en-1-H}
H(r)\sim c^2\left(\log(1-r)\right)^2, \quad \text{ as } r\to 1.
\end{equation}
\end{lem}

\begin{proof} It is straightforward to check that, for any $c>0$, $\overline h(r)\coloneqq c^2\left(\log(1-r)\right)^2$ is a strict supersolution on $(0,1)$ for the ODE, meaning $\overline h\,'>\frac{2c}{1-r}\sqrt{\overline h\,^+}-2r$ on $(0,1)$. On the other hand, from Lemma \ref{lem:sub-super}, $\underline h(r)\coloneqq (\lambda^+)^2r^2(1-r)^2$
is a strict subsolution on $(0,1)$ for the ODE  meaning $\underline h\,'<\frac{2c}{1-r}\sqrt{\underline h\,^+}-2r$ on $(0,1)$. One can check that $\underline h < \overline h$ on $(0,1)$,  and the conclusion then follows from the classical monotone iteration method. To be more precise,  we construct a sequence $(h_n)_{n\geq 0}$ of functions on $[0,1)$ by letting $h_0\equiv \underline{h}$ and
$$
h_{n+1}'(r)=\frac{2c}{1-r}\sqrt{ h_n^+}-2r, \quad h_{n+1}(0)=0,
$$
or, equivalently,
\begin{equation}
\label{h_n-integrale}
h_{n+1}(r)=2c\int_0^r \frac{\sqrt{h_n^+(s)}}{1-s}\,ds-r^2, \quad 0\leq r<1.
\end{equation}
Using that $\underline h$ and $\overline h$ are strict sub- and supersolutions, one can prove by induction that 
$$
\underline{h}\leq h_n\leq h_{n+1}\leq \overline{h}, \quad \text{for all } n\geq 0.
$$As a result, there is a function $H$ on $[0,1)$ such that, for all $r\in [0,1)$, $h_n(r)\to H(r)$ as $n\to +\infty$. By the monotone convergence theorem, we can pass to the limit in the integral formulation \eqref{h_n-integrale} and reach the first conclusion.

Now, the constructed solution satisfies $H(r)\geq \underline h(r)=(\lambda^+)^2r^2(1-r)^2>\frac{r^2(1-r)^2}{c^2}$ so that, by the equation, $H'(r)>0$. Hence $H(1)$ exists in $(0,+\infty]$. If $H(1)=\ell<+\infty$, then $H'(r)\sim \frac{2c\sqrt \ell}{1-r}$ as $r\to 1$ so that, by integration, $H(1)=+\infty$, a contradiction. Hence $H(1)=+\infty$ and $\frac{H'(r)}{2\sqrt{H(r)}}\sim \frac{c}{1-r}$ as $r\to 1$ so that, by integration, we obtain \eqref{eq-en-1-H}.
\end{proof}

Hence, for any $c\geq 2$, the solution $H$ provided by Lemma \ref{lem:large} solves problem \eqref{h-Cauchy} (and, thus, already provides {\it one} saturated travelling wave, as detailed in subsection \ref{ss:back}).  
Furthermore, by Cauchy-Lipschitz theorem, this $H$ cannot be crossed by another solution and acts as a barrier.

In the sequel, a $(r,h)$-phase plane analysis provides valuable information. Observe first that, by the equation, the $r$-axis can only be crossed downhill. Next,  the bell-shaped curve
\begin{equation}
\label{bell}
B(r)\coloneqq\frac{r^2(1-r)^2}{c^2}, \quad 0\leq r\leq 1,
\end{equation}
plays a central role. Indeed, a solution $h$ decreases (resp. increases) when it is below (resp. above) the bell. 
Also, by using and differentiating the equation, we see that  if a solution $h$ touches the bell at some $0<r<\frac 12$ (resp. $\frac 12<r<1$) then  $h'(r)=0$, $h''(r)<0$ (resp. $h'(r)=0$, $h''(r)>0$) and therefore, since $B'(r)>0$ (resp. $B'(r)<0$), the graph of  $h$ does cross the bell-shaped graph of $B$ at $r$. Notice however that if a solution $h$ touches the bell at $r=\frac 12$ then $h''(\frac 12)=0$ and the bell is not crossed.

Equipped with the barrier $H$,  we can now construct a \lq\lq small'' solution, in the sense that $h(1)=0$, corresponding to the non-saturated wave.

\begin{lem}[A unique small solution]\label{lem:small-sol} Assume $c\geq 2$. Then there is a unique solution $h_0$  to the Cauchy problem \eqref{h-Cauchy} that further satisfies $h_0(1)=0$ (and it is positive on $(0,1)$). Also, there is $r_0\in (0,\frac 12)$ such that $h_0'>0$ on $(0,r_0)$ while $h_0'<0$ on $(r_0,1)$, and $h''(r_0)<0$. Furthermore, 
\begin{equation}
\label{vers-zero-h0}
\frac{r^2(1-r)^2}{c^2}<h_0(r)<(\lambda^-)^2r^2(1-r)^2, \quad 0<r<r_0,
\end{equation}
and there is $\varepsilon>0$ small enough such that
\begin{equation}
\label{vers-un}
0<h_0(r)<\frac{r^2(1-r)^2}{c^2}, \quad 1-\varepsilon<r<1.
\end{equation}
\end{lem}

\begin{proof} If there are two such solutions, $h_1$ and $h_2$, then the boundary conditions $(h_2-h_1)(0)=(h_2-h_1)(1)=0$ enforce the existence of $r_0\in(0,1)$ such that $(h_2-h_1)'(r_0)=0$. 
From the equation this implies $h_1(r_0)=h_2(r_0)$ and then $h_1\equiv h_2$ by Cauchy-Lipschitz theorem.

As for the existence, let us consider an increasing sequence $(r_n)_{n\geq 1}$ converging to 1. Denote $h_n$ a solution to the Cauchy problem 
 \begin{equation*}
\begin{cases} \displaystyle \frac{dh_n}{dr}=\frac{2c}{1-r}\sqrt{h_n^+} -2r \quad 0<r<1, \vspace{5pt}\\
h_n(r_n)=0.
\end{cases}
\end{equation*}
Since the $r$-axis can only be crossed downhill and $h_n'(r_n)=-2r_n<0$, we have $0<h_n(r)<H(r)$ for all $r\in(0,r_n)$, so that $h_n(0)=0$. Also, on the right of $r_n$, there holds $h_n(r)=r_n^2-r^2$ for all $(r_n,1]$. Since two solutions can only cross on the $h=0$ axis, we also have $ h_n(r)<h_{n+1}(r)<H(r)$ for all $r\in(0,1]$.  As a result there is a function $h_0$ on $[0,1]$ such that, for all $r\in[0,1]$,  $h_n(r)\to h_0(r)$ as $n\to+\infty$. In particular $h_0(0)=h_0(1)=1$ and, as in the proof of Lemma \ref{lem:large}, $h_0$ solves the ODE on $(0,1)$.

From the phase plane analysis, $h_0$ has to be below the bell \eqref{bell} in a neighborhood of 1, thus providing \eqref{vers-un}, has to cross the bell only once, at some $r_0\in(0,\frac 12)$, thus providing the sign of $h_0'$. As for \eqref{vers-zero-h0}, it follows again from the phase plane analysis and the fact that the right hand side in \eqref{vers-zero-h0} is a strict subsolution  that can only be crossed in one sense and that stands strictly above the bell (since $\lambda^->\frac 1c$).
\end{proof}

Hence, for any $c\geq 2$, the solution $h_0$ provided by Lemma \ref{lem:small-sol}  solves problem \eqref{h-Cauchy} (and, thus, already provides the non-saturated travelling wave, as detailed in subsection \ref{ss:back}).  Furthermore, this $h_0$ acts as a barrier for other possible solutions.

Equipped with $h_0$ and $H$, we can  now construct infinitely many \lq\lq large'' solutions (providing infinitely many saturated waves). To do so, for a shooting parameter $\alpha>0$, we consider the Cauchy problem 
 \begin{equation}\label{cauchy-un-demi}
\begin{cases} \displaystyle \frac{dh}{dr}=\frac{2c}{1-r}\sqrt{h^+} -2r \quad 0<r<1, \vspace{5pt}\\
h(\frac 12)=\alpha,
\end{cases}
\end{equation}
whose solution is denoted by $h_\alpha$. If $\alpha<h_0(\frac 12)$ then, from the above, $h_\alpha$ cannot solve problem \eqref{h-Cauchy} while being nonnegative. If $\alpha=h_0(\frac 12)$ then $h_\alpha \equiv h_0$ the so-called small solution. If $\alpha=H(\frac 12)$ then $h_\alpha\equiv H$ the so-called large solution constructed above. From using the solution to the ODE starting from any $h(0)>0$  as a barrier, or the upper bound in \eqref{log-dessus}, we see that 
$$
\alpha_{max}\coloneqq\sup \, \{ \alpha>0: h_\alpha(0)=0\}<+\infty,
$$
and that, using again monotony arguments as in the proof of Lemma \ref{lem:large}, the supremum is reached, that is  $h_{\alpha_{max}}(0)=0$. Note also that
$$
\alpha_{max}\geq H(1/2)>\frac{(\lambda^+)^2}{16}.
$$

\begin{lem}[Infinitely many large solutions]\label{lem:other-large} Assume $c\geq 2$. Let
\begin{equation}\label{shoot-alpha}
 h_0(1/2)<   \alpha \leq \alpha_{max}.
\end{equation}
Then the solution $h=h_\alpha$ to the Cauchy problem \eqref{cauchy-un-demi} also solves the Cauchy problem \eqref{h-Cauchy}, $h>0$ on $(0,1)$,
\begin{equation}\label{eq-en-1}
h(r)\sim c^2\left(\log(1-r)\right)^2 \quad \text{ as } r\to 1,
\end{equation}
and 
\begin{equation}
\label{vers-zero-part-1}
\frac{r^2(1-r)^2}{c^2}<h(r), \quad 0<r<\varepsilon,
\end{equation}
for some $\varepsilon>0$. Furthermore, the following holds.
\begin{enumerate}
\item [(a)] If $h_0(\frac 12)<\alpha<\frac{1}{16c^2}$, then there are $0<r_2<\frac 12<r_1<1$ such that $h'>0$ on $(0,r_2)\cup(r_1,1)$ and $h'<0$ on $(r_2,r_1)$.
\item [(b)] If $\alpha=\frac{1}{16c^2}$, then $h'>0$ on $(0,\frac 12)\cap(\frac 12, 1)$ and $h'(\frac 12)=0$.
\item [(c)] If $\frac{1}{16 c^2}<\alpha\leq \alpha_{max}$, then $h'>0$ on $(0,1)$. 
\end{enumerate}
\end{lem}

\begin{proof}  The proof is a direct consequence of the above considerations, including the phase plane analysis (note that $\frac{1}{16c^2}=B(1/2)$, the \lq\lq top of the bell \eqref{bell}'').
\end{proof}

Some of the results that have been shown up to now are illustrated in Figure \ref{fig:hrplane}, where a few solutions are drawn together with the main reference curves that we used in our analysis.

\begin{figure}
    \centering
    \includegraphics[scale=.75]{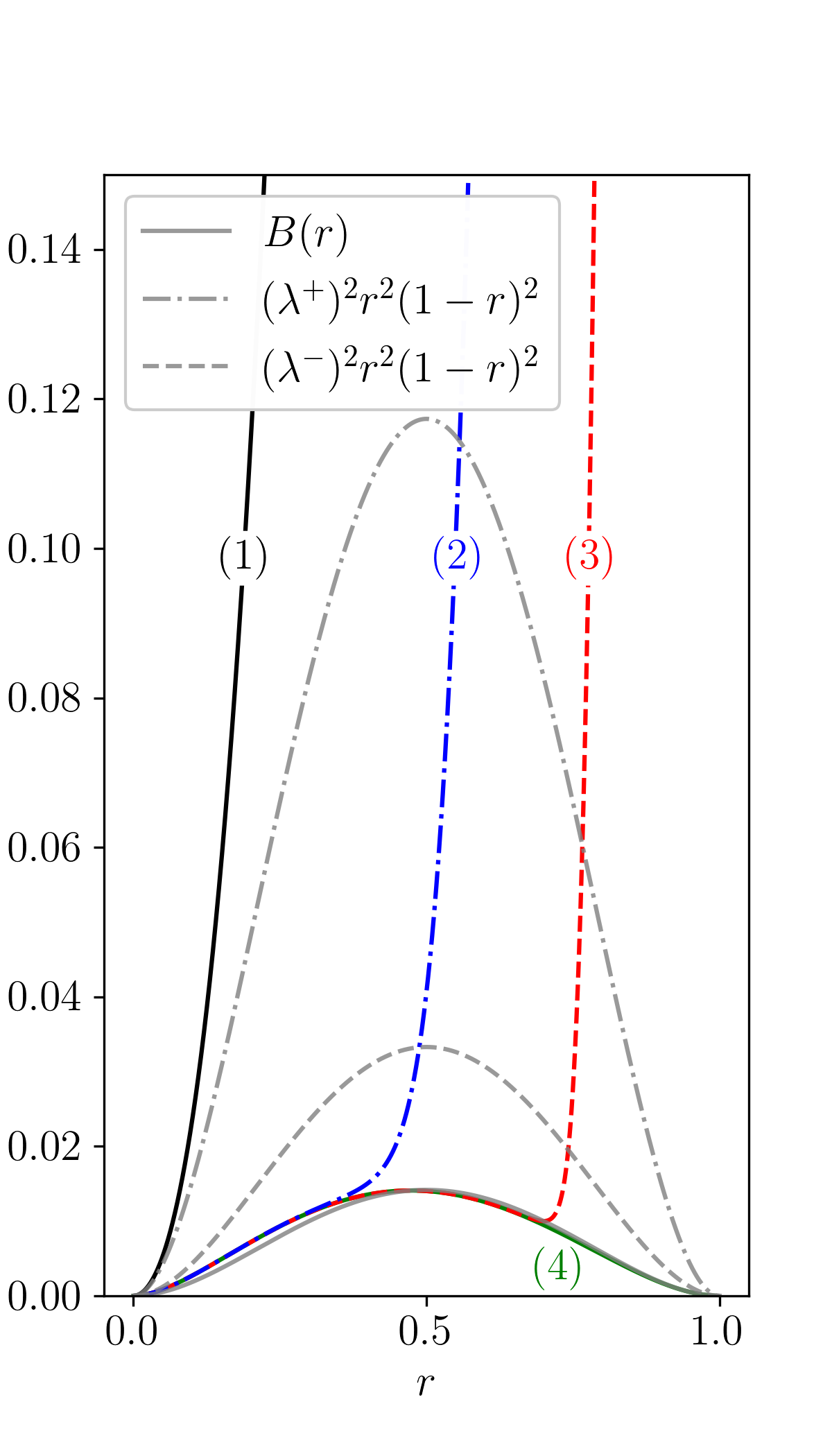}
    \includegraphics[scale=.75]{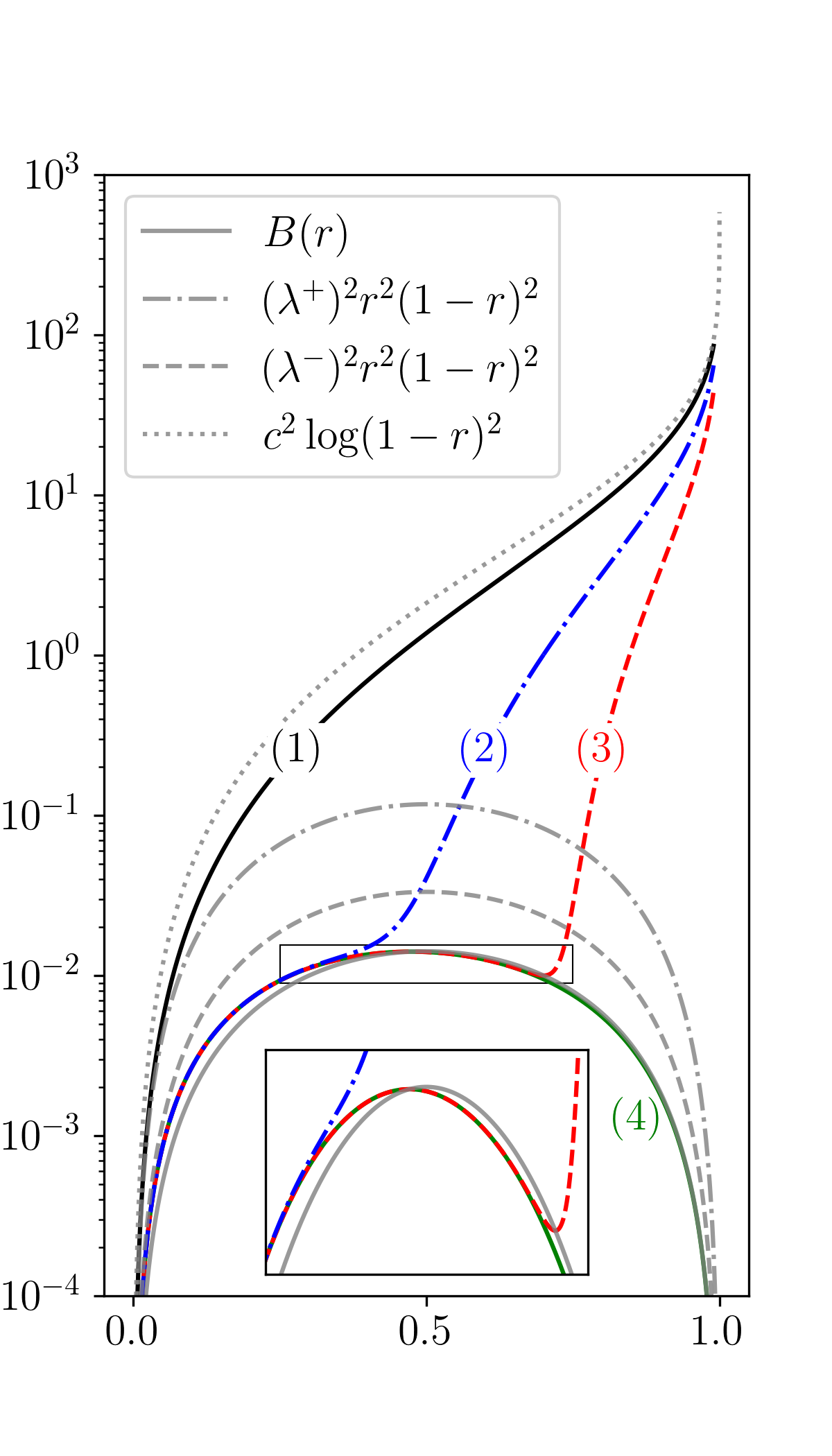}
    \caption{ The $(r,h)$-phase plane in linear (left) and log-scale (right) for $c=2.1$. The black, blue and red lines, labelled $(1)$, $(2)$ and $(3)$, correspond to three different saturated wave solutions and are computed numerically. The green line, labelled (4), corresponds to the unique non-saturated wave. In particular, with respect to Lemma \ref{lem:other-large}, $(1)$ and $(2)$ are two solutions of type $(c)$, whereas $(3)$ is a solution of type $(a)$. With respect to Lemma \ref{lem:behavior-zero}, 
    $(1)$ is a solution of type $(c)$, whereas $(2)$ and $(3)$ are solutions of type $(a)$. These waves  are represented in physical space in Figure \ref{fig:waves}.    }
    \label{fig:hrplane}
\end{figure}

\subsection{Boundary behaviors of the $h$'s}\label{ss:boundary}

Hence, there is a one-to-one correspondence between 
travelling waves $u$ (say normalized by $u(0)=\frac 12$) and the $h_\alpha$'s, with $h_0(1/2)\leq \alpha\leq \alpha_{max}$, provided by subsection \ref{ss:Cauchy}. 

Before returning to $u$, we precise the behavior of these $h_\alpha$'s as $r\to 0$ (which will translate later into an estimate on the tail of $u$ as $z\to +\infty$). This requires a combination of the phase plane analysis, some sub- and supersolutions, and some bootstrapping arguments.

\begin{lem}[Not below $\lambda^-$ as $r\to 0$]\label{lem:not-below-lambda-moins} Assume $c\geq 2$. Take $h=h_\alpha$ a solution provided by Lemma \ref{lem:small-sol} or Lemma \ref{lem:other-large}. Then, for any $0<\varepsilon<\lambda^-$, there is $r_\varepsilon\in(0,1)$ such that
$$
h(r)\geq (\lambda^--\varepsilon)^2r^2(1-r)^2, \quad \forall r\in(0,r_\varepsilon).
$$
\end{lem}
\begin{proof}Let $0<\varepsilon \ll 1$ be fixed and let $r_0\in(0,1)$ be such that the smallest root of $(1-r_0)^2 X^2-cX+1$ is larger than $\lambda^--\varepsilon$. Up to taking a smaller $r_0$, by Lemma~\ref{lem:sub-super},  $r\to (\lambda^--\varepsilon)^2r^2(1-r)^2$ is also a strict supersolution on $(0,r_0)$.  Now, assume by contradiction that there is $0<r_1<r_0$ such that
\[
    h(r)< \left(\lambda^--\eps\right)^2r^2(1-r)^2, \quad \forall r\in(0,r_1).
\]
Using the integral formulation of the Cauchy problem, one can bootstrap and improve this bound to $0<h(r)< K_n ^2 r^2(1-r)^2$, for all $n\in \NN$, where the sequence $(K_n)_{n\in \NN}$ is nonnegative, decreasing, and given by the induction
\[
K_{n+1}^2 = \frac{cK_n-1}{(1-r_0)^2}\,,\quad K_0 = \lambda^--\eps.
\]
However such a sequence cannot converge since there is no root to $(1-r_0)^2X^2-cX+1$ on the left of $K_0$, hence a contradiction. Therefore for any $0<r_1<r_0$, there is $r_\eps\in(0,r_1)$ such that $\left(\lambda^--\eps\right)^2r_\eps^2(1-r_\eps)^2\leq h(r_\eps)$. As $r\to (\lambda^--\varepsilon)^2r^2(1-r)^2$  is a strict supersolution on $(0,r_\eps)$ (see Lemma~\ref{lem:sub-super}), the lemma is proved. 
\end{proof}

\begin{lem}[Not above $\lambda^+$ as $r\to 0$]\label{lem:not-above-lambda-plus}
Assume $c\geq 2$. Take $h=h_\alpha$ a solution provided by Lemma  \ref{lem:other-large}. Then, for any $\varepsilon>0$, there is $r_\varepsilon\in(0,1)$ such that
$$
 h(r)\leq  (\lambda^++\varepsilon)^2r^2(1-r)^2, \quad \forall r\in(0,r_\varepsilon).
$$
\end{lem}

\begin{proof} 
First observe that there is $r_1>0$ such that \eqref{log-dessus} provides 
\[
h(r)\leq (2c)^2r^2(1-r)^2,\quad \forall r\in (0,r_1).
\]
Second let $0<\varepsilon \ll 1$ be fixed and let $r_0\in(0,1)$ be such that the largest root of $(1-r_0)^2X^2-cX+1$ is smaller than $\lambda^++\varepsilon$. 
Let $r_\eps = \min(r_0,r_1)$. A bootstrap argument as in the previous lemma improves the first bound inductively. Namely, for any $n\in \NN$ there holds $h(r)\leq K_{n}^2r^2(1-r)^2$ for all $r\in(0,r_\varepsilon)$, where the sequence $(K_n)_{n\in \NN}$ is positive, decreasing and  given by the induction
\[
K_{n+1}^2 = \frac{cK_n-1}{(1-r_0)^2}\,,\quad K_0 = 2c.
\]
Since $K_n$ converges to a root of $(1-r_0)^2 X^2-cX+1$ (which is smaller than $\lambda^++\varepsilon$), the lemma is proved.
\end{proof}

We now consider the solutions $h$ entering the region in-between the two bell-shaped subsolutions $r\mapsto \lambda^\pm r^2(1-r)^2$.

\begin{lem}[Sticking to $\lambda^-$ as $r\to 0$]\label{lem:sticking-lambda-moins} Assume $c\geq 2$. Take $h=h_\alpha$ a solution provided by Lemma \ref{lem:small-sol} or Lemma \ref{lem:other-large}. Assume
\begin{equation}\label{un-pt-inter-cloches}
\exists r_0\in(0,1/2],\; (\lambda^-)^2 r_0^2(1-r_0)^2\leq h(r_0)\leq (\lambda^+)^2 r_0^2(1-r_0)^2.
\end{equation}
Then, for any $0<\varepsilon<\lambda^-$, there is $r_\varepsilon\in(0,1)$ such that
$$
 (\lambda^--\varepsilon)^2r^2(1-r)^2\leq h(r)\leq  (\lambda^-+\varepsilon)^2r^2(1-r)^2, \quad \forall r\in(0,r_\varepsilon).
$$
\end{lem}

\begin{proof} Let $0<\varepsilon \ll 1$ be fixed.  The lower bound comes from Lemma \ref{lem:not-below-lambda-moins}. Next, we recall that for any $\lambda \in[\lambda^-,\lambda^+]$, $r\mapsto \lambda^2 r^2(1-r)^2$ is a strict subsolution on $(0,1)$. Hence if $h(r_0)=(\lambda^-)^2r_0^2(1-r_0)^2$ (which is necessarily the case for $c=2$ due to \eqref{un-pt-inter-cloches}, since in this case $\lambda^+=\lambda^-$) then $h(r)\leq (\lambda ^-)^2r^2(1-r)^2$ on $(0,r_0)$ and we are already done. On the other hand, if $(\lambda^-)^2 r_0^2(1-r_0)^2 < h(r_0)\leq (\lambda^+)^2 r_0^2(1-r_0)^2$, then
\begin{equation}
    \label{a-contredire}
\exists \lambda\in(\lambda^-,\lambda^+),\quad h(r)\leq \lambda ^2 r^2(1-r)^2, \quad \forall r\in(0,r_0).
\end{equation}

Assume (by contradiction) that there is $r_\ep>0$ such that 
\[
h(r)\geq (\ep \lambda^+  + (1-\ep)\lambda^-)^2r^2(1-r)^2,\quad \forall r\in (0,r_\ep).
\]
Using this into the  equation one finds, for $r\in(0,r_\ep)$,
\begin{eqnarray*}
h(r) &=& -r^2 + 2c\int_0^r\frac{\sqrt{h(s)}}{1-s}\dd s\\
&\geq &\left(c(\ep\lambda^+  + (1-\ep)\lambda^-)-1\right)r^2\\
&=&\left(\ep(\lambda^+)^2 + (1-\ep)(\lambda^-)^2\right)r^2\\
&\geq &\left(\ep(\lambda^+)^2 + (1-\ep)(\lambda^-)^2\right)r^2(1-r)^2,
\end{eqnarray*}
a bound which is strictly better than the initial bound (by convexity). Iteratively one finds
\[
h(r)\geq \left(\lambda^+ \ep_n + (1-\ep_n)\lambda^-\right)^2r^2(1-r)^2,\quad \forall r\in(0,r_\ep),\ \forall n\in\NN,
\]
with
\[\ep_0=\ep, \quad 
\ep_{n+1}(\lambda^+-\lambda^-) + \lambda^- = \sqrt{\ep_n(\lambda^+)^2 + (1-\ep_n)(\lambda^-)^2}.
\]
The sequence $(\ep _n)$ is increasing and, in view of \eqref{a-contredire}, bounded by 1, so it converges to some $\ep\leq \ell \leq 1$. Letting $n\to+\infty$ into the above recursive equation, we see that 
\[
(\ell\lambda^++(1-\ell)\lambda^-)^2 = \ell (\lambda^+)^2 + (1-\ell) (\lambda^-)^2,
\]
which (recall $\lambda^-\neq \lambda ^+$ since  $c>2$), by convexity, enforces $\ell=1$, which contradicts \eqref{a-contredire}.

As a result, there is a sequence $r_n\to 0$ such that
$$
h(r_n)<\left(\lambda^-+\ep(\lambda^+-\lambda^-)\right)^2 r_n^2(1-r_n)^2,
$$
and the conclusion follows from the fact that 
 $r\mapsto \left(\lambda^-+\ep(\lambda^+-\lambda^-)\right)^2 r^2(1-r)^2$ is a strict subsolution on $(0,1)$.
\end{proof}

Combining the above, we get the picture described in the following lemma. In particular, note that in this lemma the large solution $H$ is a solution of type $(c)$ and the small solution $h_0$ is a solution of type $(a)$.

\begin{lem}[Behavior of $h_\alpha$ as $r\to 0$]\label{lem:behavior-zero} Assume $c\geq 2$. Take $h=h_\alpha$ a solution provided by Lemma \ref{lem:small-sol} or Lemma  \ref{lem:other-large}. Then there exist
$$\frac{(\lambda^+)^2}{16} < \alpha_{switch}^- \leq \alpha_{switch}^+ \leq \alpha_{max},
$$
with  $\alpha_{switch}^- < \alpha_{switch}^+$ if  $c\in (2, \frac{3\sqrt{2}}{2})$, such that
\begin{enumerate}
\item[(a)]  If $h_0(\frac12)\leq \alpha < \alpha_{switch}^-$  then the solution sticks to $\lambda^-$ from below in the sense that, for any $0<\ep\ll 1$, there exists $r_\varepsilon\in(0,1)$ such that 
\begin{equation}\label{stick-to-lambda-minus} (\lambda ^--\ep)^2r^2(1-r)^2\leq h(r)\leq (\lambda^-)^2r^2(1-r)^2, \quad \forall r\in (0,r_\ep);
\end{equation}
\item[(b)]  If $ \alpha_{switch}^-\leq \alpha < \alpha_{switch}^+$  then the solution sticks to $\lambda^-$  from above in the sense that, for any $0<\ep\ll 1$, there exists $r_\varepsilon\in(0,1)$ such that 
\begin{equation}\label{stick-to-lambda-minus-above} (\lambda ^-)^2r^2(1-r)^2< h(r)\leq (\lambda^-+\varepsilon)^2r^2(1-r)^2, \quad \forall r\in (0,r_\ep);
\end{equation}
\item[(c)] If  $ \alpha_{switch}^+ \leq \alpha \leq \alpha_{max}$, then the solution sticks to $\lambda^+$ in the sense that, for any $0<\ep\ll 1$, there exists $r_\varepsilon\in(0,1)$ such that 
\[ (\lambda ^+)^2r^2(1-r)^2< h(r)\leq (\lambda^++\varepsilon)^2r^2(1-r)^2, \quad \forall r\in (0,r_\ep).
\]
\end{enumerate}
\end{lem}

\begin{proof}
Let us define
$$
 \alpha_{switch}^\pm\coloneqq\sup\left\{\alpha \geq \frac{(\lambda^\pm)^2}{16}: \exists r_0\in (0, 1/2],~ h_\alpha(r_0)=(\lambda^\pm)^2r_0^2(1-r_0)^2\right\}.
$$
By continuity with respect to the parameter $\alpha$, we have  $\alpha_{switch}^\pm> (\lambda^\pm)^2/16$ and the supremum cannot be a maximum.
The lower bound in point $(a)$ follows from Lemma \ref{lem:not-below-lambda-moins}, and the upper bound from the fact that $r\mapsto (\lambda^-)^2r^2(1-r)^2$ is a strict subsolution on $(0,1)$, as in the proof of Lemma \ref{lem:sticking-lambda-moins}.
Point $(b)$ is a direct application of Lemma \ref{lem:sticking-lambda-moins} and the definition of $\alpha_{switch}^-$. 
Let us show that this assumption of point $(b)$ is not always empty, and more precisely that $\alpha_{switch}^-<\alpha_{switch}^+$ if $c\in(2,\frac{3\sqrt{2}}{2})$. For this, we consider $g^-(r) = (\lambda^-+\eps)^2r^2(1-r)^2$ and $g^+(r) = (\lambda^-)^2r^2(1+\alpha r^{\beta})$, for $\alpha>0$ and $(\lambda^{-})^{-2}-1<\beta<1$ which are respectively sub- and supersolutions by Lemma~\ref{lem:sub-super}. Moreover they are enclosed between the bells $r\to(\lambda^\pm)^2 r^2(1-r)^2$ and satisfy that $g^{-} > g^{+}$, in some interval $(0,r_\ep)$. Cauchy-Lipschitz theorem then allows to build infinitely many solutions $g^{-} > h > g^{+}$ which thus have the asymptotic behavior  \eqref{stick-to-lambda-minus-above}. These solutions correspond to $h_\alpha$ with $\alpha\geq\alpha_{switch}^-$ since $r\to(\lambda^-)^2r^2(1-r)^2$ is a subsolution, and $\alpha<\alpha_{switch}^+$ by definition of $\alpha_{switch}^+$. 
Finally, the lower bound in point $(c)$ also follows from the definition of $\alpha_{switch}^+$, whereas the upper bound is a consequence of Lemma \ref{lem:not-above-lambda-plus}.
\end{proof}

We now precise the behavior of $h_0$ as $r\to 1$ (which will translate later into an estimate on the tail of $1-u$ as $z\to -\infty$). 

\begin{lem}[Sticking to $c^{-1}$ as $r\to 1$]
\label{lem:towards-1} Assume $c\geq 2$. Take $h_0$ the solution provided by Lemma \ref{lem:small-sol}. Then for any $\ep>0$, there is $r_\ep\in(0,1)$ such that
$$
\frac{r^2(1-r)^2}{(c+\ep)^2}\leq h_0(r)<\frac{r^2(1-r)^2}{c^2}, \quad \forall r\in(r_\ep,1).
$$
\end{lem}

\begin{proof} The upper bound has already been proved in Lemma~\ref{lem:small-sol}. One can check that $r\mapsto \frac {r^2(1-r)^2}{(c+\ep)^2} $ is a strict supersolution near $r=1$, say on some $(r_\ep,1)$, so that if $\frac{r^2(1-r)^2}{(c+\ep)^2} > h_0(r)$ for some $r^*_\ep \in (r_\ep,1)$ then the same inequality holds for {\it any} $r\in(r^*_\ep,1)$. Thus let us assume (by contradiction) that
\[
 h_0(r)\leq \frac{r^2(1-r)^2}{(c+\ep)^2}, \quad \forall r\in(r^*_\ep,1).
 \]
 Using this into the integral formulation
 $$
 -h_0(r)=-(1-r^2)+2c\int _r^1\frac{\sqrt{h_0(s)}}{1-s}ds,
 $$
 we obtain $h_0(r)\geq \frac{\ep}{c+\ep}(1-r^2)\geq \frac{\ep}{c+\ep}(1-r)$ which contradicts the upper bound in \eqref{vers-un}.
\end{proof}

\subsection{Back to the wave profile $u$}\label{ss:back}

For $c \geq 2$, let us now consider one of the infinitely many $h:(0,1)\to(0,+\infty)$ solving \eqref{h-Cauchy} constructed in subsection \ref{ss:Cauchy}. Reversing the manipulations in subsection \ref{ss:reaching}, we get a travelling wave. Precisely, it follows from \eqref{der-u-moin-un} and \eqref{h-def} that 
$$
(u^{-1})'(r)=\frac{1}{u'(u^{-1}(r))}=-\frac{1}{\sqrt{h(r)}}.
$$
We may normalize the travelling wave by $u(0)=\frac 12$ to, finally, reach
\begin{equation}
    \label{w}
    w(r)\coloneqq u^{-1}(r)=\int_r^{\frac 12} \frac{ds}{\sqrt{h(s)}},\quad 0<r<1.
\end{equation}

Furthermore, the quadratic behavior of $h$ at $r\to 0$  corresponds to the exponential decay of the wave towards $0$  as $z\to +\infty$. Indeed, assume that, for some $0<\varepsilon\leq \frac 12$ and $0<a<b$, we have
$$
a^2r^2(1-r)^2\leq h(r)\leq b^2r^2(1-r)^2, \quad 0<r<\varepsilon.
$$
Then it follows from \eqref{w} that
$$
\frac{1}{b}\left(\log \frac \varepsilon{1-\ep} -\log \frac r{1-r}\right)\leq w(r)-C\leq \frac{1}{a}\left(\log \frac \varepsilon{1-\ep} -\log \frac r{1-r}\right), \quad 0<r<\varepsilon, 
$$
where $C\coloneqq\int_{\varepsilon}^{1/2}\frac{ds}{\sqrt{h(s)}}>0$. Returning to $u$ via $r=u(z)$  this is transferred to 
$$
C_1e^{-bz}\leq u(z)\leq C_2 e^{-az}, \quad z>z_0,
$$
for some large enough $z_0>0$, and some constants $C_1, C_2>0$.

As for the behavior \lq\lq on the left'', there are two possibilities. First, if 
\begin{equation}
\label{int-h-infinie}
\int_{\frac 12}^{1}\frac{ds}{\sqrt{h(s)}}=+\infty,
\end{equation}
this corresponds to a non-saturated wave (i.e. $z^*=-\infty$ in the setting of Definition \ref{defi:profile}). In view of \eqref{vers-un}, this occurs for the small solution $h_0$ of Lemma \ref{lem:small-sol}. Similarly as above, the quadratic behavior of $h$ at $r\approx 1$ informs on the exponential decay of the wave towards $1$ as $z\to -\infty$: one can check that
$$
a^2r^2(1-r)^2\leq h(r)\leq b^2r^2(1-r)^2, \quad 1-\varepsilon<r<1,
$$
is transferred to 
$$
C_1e^{bz}\leq 1-u(z)\leq C_2 e^{az}, \quad z<-z_0,
$$
for some large enough $z_0>0$, and some constants $C_1, C_2>0$. 

On the other hand, if \eqref{int-h-infinie} does not hold, this corresponds to a saturated wave reaching 1 at 
$$
-\infty<z^*=-\int_{\frac 12}^{1}\frac{ds}{\sqrt{h(s)}}<0.
$$
In view of \eqref{eq-en-1}, this occurs for the infinitely many large solutions of Lemma \ref{lem:small-sol}.
Furthermore, since these $h$'s satisfy \eqref{eq-en-1}, we have
$$
u'(z)=\frac{1}{w'(u(z))}=-\sqrt{h(u(z))}\sim c\log (1-u(z)), \quad \text{ as } z\searrow z^*,
$$
which, by integration, provides
$$
-\textrm{li} (1-u(z))\sim c(z-z^*), \quad \text{ as } z\searrow z^*,
$$
where $\textrm{li} (x)\coloneqq\int_0^x \frac{ds}{\log s}$ is the logarithmic integral function. Since $\textrm{li} (x)\sim \frac{x}{\log x}$ as $x\to 0$, this is transferred to 
\begin{equation}
\label{behavior-1-saturated-implicite}
\frac{1-u_S(z)}{-\log (1-u_S(z))} \sim c(z-z^*), \quad \text{ as } z\searrow z^*.
\end{equation}
Writing $1-u_S(z)=\frac{c(z-z^*)}{\psi(z)}$ with $\psi(z)\to 0$ as $z\to z^*$, so that $\log \psi(z)=\log (c(z-z^*))-\log(1-u_S(z))$, and then using   \eqref{behavior-1-saturated-implicite}, one reaches $\frac 1{\psi(z)}\sim -\log(z-z^*)$, and thus \eqref{behavior-1-saturated}.  Note in particular that $u'(z^*)=-\infty$ and the wave is sharp. 

One may also check that, for all $z\in(z^*,+\infty)$,
$$
u''(z)=\frac 12 h'(u(z)),
$$
so that a zero for $h'$ corresponds to a zero for $u''$.  Furthermore, if $z\in(z^*,+\infty)$ is such that $u''(z)=0$ then $u'''(z)=\frac 1 2 u'(z) h''(u(z))$, which is not zero if we further assume $u(z)\neq \frac 12$ (meaning that $h$ is not the solution tangent at the top of the bell). In other words, if $u(z)\neq \frac 12$ and $u''(z)=0$, then $z$ is an inflection point.

\medskip

\begin{proof}[Conclusion]
Based on these observations, it is straightforward to check that the many $h:(0,1)\to(0,+\infty)$ solving \eqref{h-Cauchy} constructed in subsection \ref{ss:Cauchy} provide the infinitely many travelling waves as stated in Theorem \ref{th:tw}. To be more precise, $u_{NS}$ in Theorem \ref{th:tw} $(i)$ comes from the small solution of Lemma \ref{lem:small-sol}, while the $u_S$'s in Theorem \ref{th:tw} $(ii)$  come from the large solutions of Lemma \ref{lem:other-large}. As for the estimates of the convergence rate to 0 or 1, they are provided by the analysis in subsection \ref{ss:boundary}.
\end{proof}
 
\begin{rema}
 Let us finally precise that the estimate \eqref{decay-0-saturated-heavy} in Theorem~\ref{th:tw} can be made slightly more precise, based on the conclusions of Lemma~\ref{lem:behavior-zero}. Indeed, among the strictly convex saturated travelling wave profiles satisfying \eqref{decay-0-saturated-heavy}, one has that infinitely many satisfy 
 \[
 C_{11}\,e^{-\lambda^-z} \leq u_{S}(z)\leq  C_{12}\, e^{-(\lambda^--\eps) z}, \quad \text{ as } z\to +\infty.
 \]
For small admissible velocities $c\in(2,\frac{2\sqrt{3}}{2})$, there are also infinitely many satisfying
 \[
 C_{13}\,e^{-(\lambda^-+\ep)z} \leq u_{S}(z)\leq  C_{14}\, e^{-\lambda^- z}, \quad \text{ as } z\to +\infty.
 \]
 \end{rema}

\noindent{\bf Acknowledgement.}  The three authors are very grateful to the anonymous referees whose raised issues and precise comments  have improved the quality and the readability of the paper. Andrea Natale and Maxime Herda warmly thank Gaël Beaunée and Pauline Ezanno for discussions about discrete distributed contact models which lead to the initial idea for this work. Matthieu Alfaro is grateful to the team RAPSODI (INRIA Lille) for its hospitality and nice atmosphere for work. Matthieu Alfaro is supported by  the ANR project DEEV ANR-20-CE40-0011-01. The authors acknowledge support from the LabEx CEMPI (ANR-11-LABX0007) and the French institute of Mathematics for Planet Earth (iMPT).

\bibliographystyle{siam}

\bibliography{bibli}

\begin{thebibliography}{10}

\bibitem{Alf-Gil-Kim-Pel-Seo-22}
{\sc M.~Alfaro, T.~Giletti, Y.-J. Kim, G.~Peltier, and H.~Seo}, {\em On the
  modelling of spatially heterogeneous nonlocal diffusion: deciding factors and
  preferential position of individuals}, J. Math. Biol., 84 (2022), p.~38.

\bibitem{All-83}
{\sc L.~J.~S. Allen}, {\em Persistence and extinction in single-species
  reaction-diffusion models}, Bull. Math. Biol., 45 (1983), pp.~209--227.

\bibitem{And-10}
{\sc F.~Andreu-Vaillo, J.~M. Maz\'{o}n, J.~D. Rossi, and J.~J. Toledo-Melero},
  {\em Nonlocal diffusion problems}, vol.~165 of Mathematical Surveys and
  Monographs, American Mathematical Society, Providence, RI; Real Sociedad
  Matem\'{a}tica Espa\~{n}ola, Madrid, 2010.

\bibitem{Aro-77}
{\sc D.~Aronson}, {\em The asymptotic speed of propagation of a simple
  epidemic}, Nonlinear diffusion, 14 (1977), pp.~1--23.

\bibitem{Aro-Wei-78}
{\sc D.~G. Aronson and H.~F. Weinberger}, {\em Multidimensional nonlinear
  diffusion arising in population genetics}, Adv. Math., 30 (1978), pp.~33--76.

\bibitem{Atk-Reu-76}
{\sc C.~Atkinson and G.~E.~H. Reuter}, {\em Deterministic epidemic waves},
  Math. Proc. Cambridge Philos. Soc., 80 (1976), pp.~315--330.

\bibitem{Bar-77}
{\sc A.~D. Barbour}, {\em The uniqueness of {A}tkinson and {R}euter's epidemic
  waves}, Math. Proc. Cambridge Philos. Soc., 82 (1977), pp.~127--130.

\bibitem{Ber-Roq-Ros-21}
{\sc H.~Berestycki, J.-M. Roquejoffre, and L.~Rossi}, {\em Propagation of
  epidemics along lines with fast diffusion}, Bulletin of Mathematical Biology,
  83 (2021), p.~2.

\bibitem{Ber-Dal-Ugh-90}
{\sc M.~Bertsch, R.~Dal~Passo, and M.~Ughi}, {\em Discontinuous “viscosity”
  solutions of a degenerate parabolic equation}, Transactions of the American
  Mathematical Society, 320 (1990), pp.~779--798.

\bibitem{Ber-Ugh-90}
{\sc M.~Bertsch and M.~Ughi}, {\em Positivity properties of viscosity solutions
  of a degenerate parabolic equation}, Nonlinear Anal., 14 (1990),
  pp.~571--592.

\bibitem{Bes-Fay-21}
{\sc C.~Besse and G.~Faye}, {\em Spreading properties for sir models on
  homogeneous trees}, Bulletin of Mathematical Biology, 83 (2021), pp.~1--27.

\bibitem{Bro-Car-77}
{\sc K.~J. Brown and J.~Carr}, {\em Deterministic epidemic waves of critical
  velocity}, Math. Proc. Cambridge Philos. Soc., 81 (1977), pp.~431--433.

\bibitem{Bur-00-book}
{\sc R.~B\"urger}, {\em The mathematical theory of selection, recombination,
  and mutation}, Wiley Series in Mathematical and Computational Biology, John
  Wiley \& Sons, Ltd., Chichester, 2000.

\bibitem{Bur-Duc-Gri-23}
{\sc J.-B. Burie, A.~Ducrot, and Q.~Griette}, {\em Asymptotic behavior of an
  epidemic model with infinitely many variants}, Journal of Mathematical
  Biology, 87 (2023), p.~40.

\bibitem{Dal-Luc-87}
{\sc R.~Dal~Passo and S.~Luckhaus}, {\em A degenerate diffusion problem not in
  divergence form}, J. Differential Equations, 69 (1987), pp.~1--14.

\bibitem{Die-78}
{\sc O.~Diekmann}, {\em Thresholds and travelling waves for the geographical
  spread of infection}, Journal of Mathematical Biology, 6 (1978),
  pp.~109--130.

\bibitem{Dra-Tak-21}
{\sc P.~Dr{\'a}bek and P.~Tak{\'a}{\v{c}}}, {\em Travelling waves in the
  {Fisher}-{KPP} equation with nonlinear degenerate or singular diffusion},
  Appl. Math. Optim., 84 (2021), pp.~1185--1208.

\bibitem{Duc-20}
{\sc R.~Ducasse}, {\em Qualitative properties of spatial epidemiological
  models}, arXiv preprint arXiv:2005.06781,  (2020).

\bibitem{Duc-22}
\leavevmode\vrule height 2pt depth -1.6pt width 23pt, {\em Threshold phenomenon
  and traveling waves for heterogeneous integral equations and epidemic
  models}, Nonlinear Analysis, 218 (2022), p.~112788.

\bibitem{Duc-Nor-23}
{\sc R.~Ducasse and S.~Nordmann}, {\em Propagation properties in a
  multi-species sir reaction-diffusion system}, Journal of Mathematical
  Biology, 87 (2023), p.~16.

\bibitem{Duc-Gil-14}
{\sc A.~Ducrot and T.~Giletti}, {\em Convergence to a pulsating travelling wave
  for an epidemic reaction-diffusion system with non-diffusive susceptible
  population}, Journal of mathematical biology, 69 (2014), pp.~533--552.

\bibitem{Eng-85}
{\sc H.~Engler}, {\em Relations between travelling wave solutions of
  quasilinear parabolic equations}, Proc. Amer. Math. Soc., 93 (1985),
  pp.~297--302.

\bibitem{Eng-Gav-San-13}
{\sc R.~Engui\c{c}a, A.~Gavioli, and L.~Sanchez}, {\em A class of singular
  first order differential equations with applications in reaction-diffusion},
  Discrete Contin. Dyn. Syst., 33 (2013), pp.~173--191.

\bibitem{Fay-Roq-Zha-23}
{\sc G.~Faye, J.-M. Roquejoffre, and M.~Zhang}, {\em Spreading properties in
  {K}ermack-{M}c{K}endrick models with nonlocal spatial interactions--{A} new
  look}, arXiv preprint arXiv:2304.11873,  (2023).

\bibitem{Fis-37}
{\sc R.~A. Fisher}, {\em The wave of advance of advantageous genes}, Ann.
  Eugen., 7 (1937), pp.~355--369.

\bibitem{Hos-Ily-95}
{\sc Y.~Hosono and B.~Ilyas}, {\em Traveling waves for a simple diffusive
  epidemic model}, Mathematical Models and Methods in Applied Sciences, 5
  (1995), pp.~935--966.

\bibitem{Ken-65}
{\sc D.~G. Kendall}, {\em Mathematical models of the spread of infection},
  Mathematics and computer science in biology and medicine,  (1965),
  pp.~213--225.

\bibitem{Ker-Mac-27}
{\sc W.~O. Kermack and A.~G. McKendrick}, {\em A contribution to the
  mathematical theory of epidemics}, Proceedings of the royal society of
  london. Series A, Containing papers of a mathematical and physical character,
  115 (1927), pp.~700--721.

\bibitem{Kol-Pet-Pis-37}
{\sc A.~Kolmogorov, I.~Petrovskii, and N.~Piskunov}, {\em {\'E}tude de
  l’equations de la chaleur, de la mati{\'e}re et son application {\'a} un
  probleme biologique}, Bull. Moskov. Gos. Univ. Mat. Mekh, 1 (1937), p.~125.

\bibitem{Mal-Mar-03}
{\sc L.~Malaguti and C.~Marcelli}, {\em Sharp profiles in degenerate and doubly
  degenerate {Fisher}-{KPP} equations.}, J. Differential Equations, 195 (2003),
  pp.~471--496.

\bibitem{Med-Kot-03}
{\sc J.~Medlock and M.~Kot}, {\em Spreading disease: {Integro}-differential
  equations old and new.}, Math. Biosci., 184 (2003), pp.~201--222.

\bibitem{Mol-72}
{\sc D.~Mollison et~al.}, {\em The rate of spatial propagation of simple
  epidemics}, in Proceedings of the sixth Berkeley symposium on mathematical
  statistics and probability, vol.~3, Berkeley and Los Angeles, University of
  California Edinburgh, UK, 1972, pp.~579--614.

\bibitem{San-Mai-94}
{\sc F.~S{\'a}nchez-Garduno and P.~K. Maini}, {\em Existence and uniqueness of
  a sharp travelling wave in degenerate non-linear diffusion {Fisher}-{KPP}
  equations}, J. Math. Biol., 33 (1994), pp.~163--192.

\bibitem{San-Mai-95}
\leavevmode\vrule height 2pt depth -1.6pt width 23pt, {\em Traveling wave
  phenomena in some degenerate reaction-diffusion equations}, J. Differential
  Equations, 117 (1995), pp.~281--319.

\bibitem{Thi-77}
{\sc H.~Thieme}, {\em A model for the spatial spread of an epidemic}, Journal
  of Mathematical Biology, 4 (1977), pp.~337--351.

\bibitem{Ugh-84}
{\sc M.~Ughi}, {\em A degenerate parabolic equation modelling the spread of an
  epidemic}, Ann. Mat. Pura Appl. (4), 143 (1986), pp.~385--400.

\bibitem{Wan-Yin-03}
{\sc C.~Wang and J.~Yin}, {\em Traveling wave fronts of a degenerate parabolic
  equation with non-divergence form}, J. Partial Differ. Equations, 16 (2003),
  pp.~62--74.

\bibitem{Yin-Jin-09}
{\sc J.~Yin and C.~Jin}, {\em Critical exponents and traveling wavefronts of a
  degenerate-singular parabolic equation in non-divergence form}, Discrete
  Contin. Dyn. Syst. Ser. B, 13 (2009), pp.~213--227.

\bibitem{Yin-Jin-09-bis}
\leavevmode\vrule height 2pt depth -1.6pt width 23pt, {\em Travelling
  wavefronts for a non-divergent degenerate and singular parabolic equation
  with changing sign sources}, Proc. Roy. Soc. Edinburgh Sect. A, 139 (2009),
  pp.~1179--1207.

\end{thebibliography}

\end{document}